\documentclass[12pt, reqno]{amsart}
\usepackage[english]{babel}
\usepackage{amsmath, amsthm, amscd, amsfonts, amssymb, graphicx, color}
\usepackage[bookmarksnumbered, colorlinks, plainpages=false]{hyperref}
\usepackage{csquotes}
\usepackage[utf8]{inputenc}
\usepackage{amsthm}
\usepackage{amsmath}
\usepackage{pgf,tikz}
\usepackage{amsfonts,amssymb,enumerate}
\newtheorem{Defi}{Definition}
\newtheorem{Def}{Definition}[subsection]

\newtheorem{prop}[Def]{Proposition}
\newtheorem{lemma}[Def]{Lemma}

\newtheorem{introthm}{Theorem}

\newtheorem{introcor}[introthm]{Corollary}

\newtheorem{introprop}[introthm]{Proposition}

\DeclareMathOperator\Ima{Im}
\DeclareMathOperator\Ree{Re}
\DeclareMathOperator\diag{diag}
\newcommand{\R}{\mathbb{R}}

\newcommand{\C}{\mathbb{C}}

\newcommand{\p}{\mathbb{P}}
\newcommand{\Z}{\mathbb{Z}}
\newcommand{\F}{\mathbb{F}}
\newcommand{\Or}{\mathcal{O}}

\newcommand{\tor}{\mathfrak{t}}

\newcommand{\lied}{\mathcal{L}}

\everymath{\displaystyle}
\usepackage[backend=biber, style=numeric, sorting=nyt]{biblatex}
\title{SU(n)-structures through quotient by torus actions}
\author{Quentin Peres}
\begin{document}
\address{Quentin Peres \\Institut Camille Jordan, Université Claude Bernard Lyon 1 \\ 69100 Villeurbanne,France}
\email{peres@math.univ-lyon1.fr}
\setcounter{page}{1}
\maketitle
\begin{abstract} We show that if $(X,g,J,\omega)$ is a Kähler manifold with an $SU(n+s)$-structure and a Hamiltonian holomorphic action of a compact torus $T^s$, then the usual symplectic quotient $Y$ inherits an $SU(n)$-structure provided the existence of special $1$-forms on $X$, called twist forms. We then give several applications of our results: on complex projective spaces, on cones over Fano Kähler-Einstein manifold and on toric $\C\p^1$ bundles. We also study the geometry behind these structures in the case of $n=3$.
\end{abstract} 
\section{Introduction}
 String theory requires, in addition to the four-dimensional spacetime, a six-dimensional compact manifold, called the compactification space. To obtain realistic particle physics model, this space should admit an $SU(3)$-structure with torsion \cite{Gra_a_2006}. Such structures were considered in \cite{EellsSalamon} where Eells and Salamon use the Levi-Civita connection of $\C\p^3$ to transform the standard integrable complex structure of $\C\p^3$ in a non-integrable complex structure. In \cite{Larfors_2010,Larfors_2013}, Larfors, Lüst and Tsimpis also study the question of producing $SU(3)$-structures with torsion on $6$-dimensional toric manifolds.\\
\indent Let us recall some basic facts about $SU(n)$-structures. A smooth manifold $X$ of real dimension $2n$ admits an $SU(n)$-structure if $X$ is an almost Hermitian manifold with a complex volume form, i.e $X$ has a Hermitian metric $h$, an almost complex structure $J$, an almost symplectic form $\omega$ which is $(1,1)$ with respect to $J$, such that $h,\omega$ are $J$-invariant with $\omega=-\Ima(h)$, and $\Omega$ a nowhere-vanishing complex volume form of type $(n,0)$. Equivalently, $X$ admits an $SU(n)$-structure if :
\begin{enumerate}[i)]
    \item There exists a nondegenerate real $2$ form $\omega$.
    \item There exists a locally decomposable nowhere-vanishing complex $n$ form $\Omega$.
    \item $\Omega,\omega$ satisfy the $SU(n)$-equations:
    \begin{align}
        \Omega\wedge\omega&=0\label{SU(n)eq1},\\
        \Omega\wedge\overline{\Omega}&=c_n\omega^n,\label{SU(n)eq2}
    \end{align}
\end{enumerate}
where $c_n=\frac{2^n}{n!}(-i)^{n^2}$ and $\omega^n=\omega\wedge\cdots\wedge\omega$ $n$-times. Given this data, we recover the almost complex structure $J$ and the Hermitian metric $h$ such that $\Omega$ is $(n,0)$ and $\omega$ is $(1,1)$ as in \cite{cpxstructuresun}
. \\
\indent In \cite{Larfors_2010}, Larfors, Lüst and Tsimpis give an explicit method to construct $SU(3)$-structures on $6$-dimensional toric manifolds. Their approach consists in using the standard $SU(n)$-structure on $\C^n$ and pushing everything down to the toric quotient by algebraic manipulations. Our goal is to generalize this construction in the framework of Kähler geometry.  
\subsection{A twist construction}
\indent Let $(X,h_X,J_X,\omega_X,\Omega_X)$ be a Kähler manifold with an $SU(n+s)$-structure given by $(\omega_X,\Omega_X)$. Consider a Hamiltonian and holomorphic action of a compact torus $T^s$ on $X$ with moment map $\mu$. We let $\tor :=Lie(T^s)$ be the Lie algebra of $T^s$ and for an element $a\in\tor$, we denote $V_a$ the vector field induced by $a$ on $X$. \\ 
\indent Denote $M$ a regular level set of $\mu$, $Y=M/\!\!/T^s$ the symplectic quotient which is an orbifold of real dimension $2n$, $i : Y\hookrightarrow M$ the inclusion map and $\pi : M\to Y$ the projection map. Starting from $(\Omega_X,\omega_X)$ the $SU(n+s)$-structure on $X$, we aim to construct $(\Omega_Y,\omega_Y)$ an $SU(n)$-structure on $Y$. For that, we shall construct basic forms on $M$ satisfying the same equations as $(\ref{SU(n)eq1})$ and $(\ref{SU(n)eq2})$.\\
\indent Descending $\Omega_X,\omega_X$ to $Y$ requires to introduce twist forms:
\begin{Defi}\label{twistformdef}
    Let $(X,J,T^s)$ be a complex manifold with an action of $T^s$. A $1$-form $\alpha$ on $X$ is a twist form if:
    \begin{enumerate}
        \item $\alpha$ is a nowhere vanishing $(1,0)$-form with respect to $J$.
        \item $\alpha$ is horizontal: $\forall a\in\tor,\alpha(V_a)=0$.
        \item $\alpha$ is charge definite : there exists a linear form $k_\alpha : \tor\to\C$ such that \[\forall a\in\tor, \lied_{V_a}\alpha=k_\alpha(a)\alpha.\]
    \end{enumerate}
    In that case, $k_\alpha$ is called the charge of $\alpha$.
\end{Defi}
 Similarly, we say that $\Omega_X$ is charge definite of charge $q:\tor\to i\R$ if
 \[\forall a\in\tor, \lied_{V_a}\Omega_X=q(a)\Omega_X\]which is always the case for $T$-invariant Calabi-Yau cones \cite{Martelli_2008}. Our main goal is to prove the following theorem:
\begin{introthm}\label{twistsimpletheo}
        Let $(X,h_X,J_X,\omega_X,\Omega_X)$ be a Kähler manifold of complex dimension $n+s$ with $(\Omega_X,\omega_X)$ an $SU(n+s)$-structure on $X$ with charge operator $q$ taking values in the pure imaginary numbers. Consider a free, Hamiltonian and holomorphic action of a compact torus $T^s$ on $X$ with moment map $\mu$ and let $M$ be a regular level set of $\mu$. Assume there exists a twist form $\alpha_X$ on $X$ of charge $k_\alpha=\frac{q}{2}$. Then, for any basis $e_1,\cdots,e_s$ of $\tor$ with induced vector fields $V_{e_i}=V_i$, the forms: \[\Omega=i^s\overline{\alpha_M}\wedge\bigg(\overline{\alpha_M}\cdot(\iota_{V_1}\cdots\iota_{V_s}\Omega_X)_M\bigg), \,\omega=\omega_M-\frac{i} {\|\alpha_M\|^2}\alpha_M\wedge\overline{\alpha_M}\] descend to $Y$ and satisfy the equations:
        \begin{align}
            \Omega\wedge\omega&=0, \label{eq1theo}\\
            \Omega\wedge\overline{\Omega}&=c_n\|\alpha_M\|^4\|V_1\|^2\cdots\|V_s\|^2\omega^n,
        \end{align}
        where $\overline{\alpha_M}\cdot$ is the contraction of forms by the complex vector field induced by the Hermitian metric $h_X$ and for a form $\Phi_X$ on $X$, $\Phi_M$ is the pullback of $\Phi_X$ on $M$. 
\end{introthm}
From that, we readily infer that:
\begin{introcor}\label{cor1}
    The toric quotient $Y=M/\!\!/T^s$ admits an $SU(n)$-structure. 
\end{introcor}
Note that the hermitian structure underlying the $SU(n)$-structure obtained on $Y$ is not the one obtained by Kähler reduction. In fact, the examples detailed in Section $4$ suggest that $\omega$ is never closed and the underlying almost complex structure never integrable. \\ \\
\noindent\textbf{Remark 1.}
    If $\alpha$ is a twist form for a torus $T$, then it is a twist form for any subtorus $T'\subset T$. In particular, Theorem \ref{twistsimpletheo} can be applied for various subtorus. \\ \\
\indent We then generalize Theorem \ref{twistsimpletheo} using multiple twist forms and an orthogonality assumption, see Definition \ref{orthogonaldef}. 
\begin{introthm}\label{multipletwisttheo}
    Let $(X,h_X,J_X,\omega_X,\Omega_X)$ be a Kähler manifold of complex dimension $n+s$ with $(\Omega_X,\omega_X)$ an $SU(n+s)$-structure on $X$ with charge operator $q$ taking values in the pure imaginary numbers. Consider a free, Hamiltonian and holomorphic action of a compact torus $T^s$ on $X$ with moment map $\mu$ and let $M$ be a regular level set of $\mu$. Let $\{\alpha^i_X\}_{1\leq i\leq l}$ be a family of orthogonal twist forms on $X$ of charges $k_{\alpha_1}+\cdots+k_{\alpha_l}=\frac{q}{2}$. Then, for any basis $e_1,\cdots,e_s$ of $\tor$ with induced vector fields $V_{e_1}=V_i$, the forms:
    \[\Omega=i^s\overline{\alpha^1}_M\wedge\cdots\wedge\overline{\alpha^l}_M\wedge\bigg(\overline{\alpha^1}_M\cdot...\cdot\overline{\alpha^l}_M\cdot(\iota_{V_1}\cdots\iota_{V_s}\Omega_X)_M\bigg), \]\[\omega=\omega_M-\sum_{k=1}^l\frac{i}{\|\alpha^k_M\|^2}\alpha^k_M\wedge\overline{\alpha^k}_M\]
    are basic and satisfy the equations:
    \begin{align}
        \Omega\wedge\omega&=0, \label{eq1theo2}\\
        \Omega\wedge\overline{\Omega}&=(-1)^{l+1}c_n\|\alpha^1_M\|^4\cdots\|\alpha^l_M\|^4\|V_1\|^2\cdots\|V_s\|^2\omega^n. \label{eq2theo2}
    \end{align}
\end{introthm}
\noindent\textbf{Remark 2.} If we choose another basis of $\tor$, then $\Omega$ is multiplied by a real constant. Similarly, if we choose another family of orthogonal twist forms with the same charges that spans the same space, then $\Omega$ is invariant. 
\subsection{Existence and space of twist forms} Many (non-compact) Kähler Hamiltonian $T$-spaces with an $SU(n)$-structure are known in the literature, for example all toric Calabi-Yau cones \cite{1264601036}, and provide potential applications of Theorems \ref{twistsimpletheo} and \ref{multipletwisttheo}. However, both theorems require to assume the existence of a twist form of appropriate charge. This leads us to address the question of existence of twist forms. We first find an explicit family of holomorphic twist forms of polynomial type on $\C^N\backslash\{0\}$ with $N=2n$ even of charge $2$ for the diagonal action of $S^1$.
\begin{introprop}\label{su(3)cp3} Let $N=2n\geq 2$ be even and $M=(m_{ij})_{1\leq i,j\leq N}$ an invertible skew-symmetric $N\times N$ complex matrix. Then, the form: \[\alpha_M:={}^tZMdZ=\sum_{i,j=1}^Nz_im_{ij}dz_j\] is a holomorphic twist form of charge $2$ on $\C^N\backslash\{0\}$ for the diagonal action of $S^1$. Conversely, any holomorphic twist form of charge $2$ on $\C^N\backslash\{0\}$ for the diagonal action of $S^1$ is of this form.
\end{introprop}
We also prove that holomorphic twist forms on $\C^N\backslash\{0\}$ of charge strictly greater than $2$ do not exist, see Proposition \ref{nonexistence}. 
\subsection{Explicit examples and torsion}
Our first application is motivated by the example of an $SU(3)$-structure on $\C\p^3$ in \cite{Larfors_2013,Larfors_2010}. We generalize this example by using the family of Proposition \ref{su(3)cp3} in the case $N=4$. Theorem \ref{twistsimpletheo} then gives a whole new family of $SU(3)$-structures on $\C\p^3$, the example of Larfors, Lüst and Tsimpis being included in this family.
\begin{introthm}\label{cp3}
     Consider $(\C^4,\Omega_0,\omega_0)$ with its canonical $SU(4)$-structure and with complex holomorphic Euler vector field $\xi$. Let $M=(m_{ij})_{1\leq i,j\leq 4}$ be a skew-symmetric invertible $4\times 4$ complex matrix. Denote \[\alpha_M:={}^tZMdZ,\] the corresponding holomorphic twist form of charge $2$ on $\C^{4}\backslash\{0\}$. Then, the forms: \[\Omega_M:=\frac{1}{\|\alpha_M\|^2}\overline{\alpha_M}\wedge(\overline{\alpha_M}\cdot(\iota_\xi\Omega_0)),\, \omega_M:=\omega_0-\frac{i}{\|\alpha_M\|^2}\alpha_M\wedge\overline{\alpha_M}\] define an $SU(3)$-structure on $\C\p^3$.
\end{introthm}
\indent The family given by Proposition \ref{su(3)cp3} suggests that we should apply Theorem \ref{multipletwisttheo} for higher dimensional projective spaces. More precisely, we want to use our family of twist forms of charge $2$ to provide to $\C\p^{4n-1}$ an $SU(4n-1)$-structure. To apply our result, we use a Gram-Schmidt process to produce an orthogonal family of twist forms out of our family. However, this produces zeros failing to satisfy the condition of nowhere-vanishing. Consequently, we can still apply Theorem \ref{twistsimpletheo} or \ref{multipletwisttheo} but in that case we will obtain smooth $SU(n)$-structures outside the zeros of our family. More precisely, we show the following result for $\C\p^7$:
\begin{introthm}\label{cp7}
    Let $\alpha_1,\alpha_2$ be two generic holomorphic twist forms of charge $2$ on $\C^8\backslash\{0\}$ with its canonical $SU(8)$-structure $(\Omega_0,\omega_0)$. Denote $\beta_1,\beta_2$ the resulting forms of the Gram-Schmidt process. Then, the following forms:
    \[\Omega:=\frac{1}{\|\beta_1\|^2\|\beta_2\|^2}\overline{\beta_1}\wedge\overline{\beta_2}\wedge(\overline{\beta_1}\cdot\overline{\beta_2}\cdot\iota_\xi\Omega_0),\,-\omega=\omega_0-\frac{i}{\|\beta_1\|^2}\beta_1\wedge\overline{\beta_1}-\frac{1}{\|\beta_2\|^2}\beta_2\wedge\overline{\beta_2},\] define a smooth $SU(7)$-structure on $\C\p^7$ outside four generic projective lines in $\C\p^7$. 
\end{introthm}
In this example, note that the orthogonal twist forms are not holomorphic. Moreover, the singularities are studied and cannot be extended continuously. 
\indent We then review the example of toric $\C\p^1$ bundles over Hirzebruch surfaces $\F_n$ for $n\in\Z$ in \cite{Larfors_2010} as it is another source of examples where the twist forms we compute are not holomorphic. Hirzebruch surfaces are toric surfaces described by a twisted action of $\C^*\times\C^*$ on $\C^2\backslash\{0\}\times\C^2\backslash\{0\}$ given by : \[(\lambda,\mu)\cdot(z_0,z_1)\times(z_0',z_1')=(\lambda z_0,\lambda z_1,\lambda^n\mu z'_0,\mu z'_1)\] In particular, $\F_0=\C\p^1\times\C\p^1$ and the quotient description with respective charges $(1,1,n,0)$ and $(0,0,1,1)$ are used to describe a toric $\C\p^1$ bundle over $\F_n$. Such a bundle is specified by the charges of the action of $\C^*\times\C^*\times\C^*$ on $\C^2\backslash\{0\}\times\C^2\backslash\{0\}\times\C^2\backslash\{0\}$ which are $q^1=(1,1,n,0,0,a), q^2=(0,0,1,1,0,b)$ and $q^3=(0,0,0,0,1,1)$ where $a,b$ are integers. We compute a non-holomorphic twist form in the case where $a=2-n$ and $b=-2$ to apply Theorem \ref{twistsimpletheo}. 
\begin{introprop}
    Let $n$ be an integer and $a=2-n$. Let $X$ be a toric $\C\p^1$ bundle over $\F_n$ with charges:
    \[q^1:=(1,1,n,0,0,a),\, q^2:=(0,0,1,1,0,-2),\, q^3:=(0,0,0,0,1,1).\] Consider $\Omega_0$ the standard holomorphic volume form on $\C^6$ which is of charge $q$ with respect to this action of $T^3$. Then, the form: \[\alpha=z_5(-z_2dz_1+z_1dz_2)+z_6\bigg(-z_4dz_3+z_3dz_4+\frac{nz_3z_4}{|z_1|^2+|z_2|^2}(\overline{z_1}dz_1+\overline{z_2}dz_2)\bigg)\] is a twist form of charge $\frac{q}{2}$. 
\end{introprop}
We then give another source of examples based on Calabi-Yau cones. It is well-known that if $Y$ is a compact Fano Kähler-Einstein manifold, then its cone $K_Y\backslash Y$ admits a Calabi-Yau structure \cite{sparks2010sasakieinsteinmanifolds}. If $Y$ is further assumed to have a Fano index multiple of $4$, then we can write:\[-K_Y=L^{\otimes 4},\] where $L$ is a holomorphic line bundle that we suppose very ample. Using Kodaira's theorem \cite{Wells1980}, we obtain a holomorphic $S^1$ equivariant embedding from $L^{-1}\backslash Y$ to a certain $\C^N\backslash\{0\}$. This embedding allows us to pullback the twist forms given by Proposition \ref{su(3)cp3} and to have the following result:
\begin{introthm}\label{cones}
    Let $(Y^n,J,\omega)$ be a compact Fano Kähler-Einstein manifold. If the Fano index of $Y$ is a multiple of $4$ and assuming $L^{\otimes 4}=-K_Y$ with $L$ very ample, then every non-trivial smooth section $\alpha$ of $(T^{1,0}Y)^*\otimes L^{\otimes 2}$ produces a smooth $SU(n)$-structure on the complement set of the zeros of $\alpha$ in $Y$.
\end{introthm}
As an application of this theorem, we give another proof showing how to endow $\C\p^{4n-1}$ with an $SU(4n-1)$-structure.\\
\indent Eventually, we explicitly compute the torsion of the family of $SU(3)$-structures that we have constructed on $\C\p^3$ and we show that some of them are LT, see \cite{Larfors_2013,Larfors_2010}. More precisely, Theorem \ref{cp3} gives a family of holomorphic twist forms parametrized by skew-symmetric invertible $4\times 4$ complex matrices. For such a matrix $M$, $\alpha_M$ is the corresponding holomorphic twist form and $(\Omega_M,\omega_M)$ is the constructed $SU(3)$-structure on $\C\p^3$. The differential of $\alpha_M$ satisfies the following decomposition: \[d\alpha_M=\lambda_M\alpha_M\cdot\Omega_M+\alpha_M\wedge\eta_M,\] 
where $\lambda_M$ is a complex function and $\eta_M$ a $(1,0)$-form. The torsion of $(\Omega_M,\omega_M)$ is then: 
\begin{introthm}\label{torsion}
    Let $M$ be a skew-symmetric invertible $4\times 4$ complex matrix and $\alpha_M$ the corresponding holomorphic twist form for the $SU(3)$-structure $(\Omega_M,\omega_M)$ on $\C\p^3$. Denote by $\lambda_M$ the smooth complex function and $\eta_M$ the $(1,0)$-form given by: \[d\alpha_M=\lambda_M\alpha_M\cdot\Omega_M+\alpha_M\wedge\eta_M.\] The torsion classes are:
    \[W_1=\frac{4}{3}\overline{\lambda_M},\]
    \begin{align*}W_2\wedge\omega_M=&\frac{W_1}{2}(\omega_M+\frac{3i}{2\|\alpha_M\|^2}\alpha_M\wedge\overline{\alpha_M})\wedge\omega_M+\frac{i\overline{\lambda_M}}{\|\alpha_M\|^2}\alpha_M\wedge\overline{\alpha_M}\wedge\omega_0+\\&\frac{1}{\lambda_M\|\alpha_M\|^2}\alpha_M\wedge\overline{\alpha_M}\wedge d\eta_M,\end{align*} 
    \[W_3=\Ree(\eta_M)\wedge(\omega_M+\frac{i}{\|\alpha_M\|^2}\alpha_M\wedge\overline{\alpha_M}),\, W_4=-\Ree(\eta_M),\, W_5=2\Ree(\eta_M).\] Moreover, the structure is LT, i.e $W_3=W_4=W_5=0$ iff there exists $\mu\in\R$ such that \[M^*M=\mu I_4,\] and in this case \[W_2=\frac{W_1}{2}(\omega_M+\frac{3i}{2\|\alpha_M\|^2}\alpha_M\wedge\overline{\alpha_M}).\] 
\end{introthm}
The paper is organized as follows: \\
\indent In Section $2$, we introduce the notations that we will be using throughout this paper and we first prove Theorem \ref{twistsimpletheo} in Section $2.1$. We then use the different computations of the proof of Theorem \ref{twistsimpletheo} to prove Theorem \ref{multipletwisttheo} by induction in Section $2.2$.\\
\indent Section $3$ is dedicated to the different geometric applications. In Section $3.1$, we study the existence of holomorphic twist forms in $\C^N$ for the diagonal circle action where we prove Propositions \ref{su(3)cp3} and \ref{nonexistence}. In the next section, we detail the construction of $SU(4n-1)$-structures on $\C\p^{4n-1}$ with a particular attention to the case of $\C\p^3$. In Section $3.2$, we study the example of toric $\C\p^1$ bundles over the Hirzebruch surfaces and detail the computations of the twist form. Section $3.3$ is dedicated to Calabi-Yau cones where we prove Theorem \ref{cones} and review the example of the projective spaces in this setting.\\
\indent In Section $4$, we recall what the torsion classes of an $SU(3)$-structure are and some basic facts about the geometry each of them encode. We then proceed to compute these classes for the $SU(3)$-structures on $\C\p^3$ we have found and discuss the type of geometry we can get with this construction. \\ \\
\textbf{Acknowledgments.} I would like to thank my supervisors Eveline Legendre and Dimitrios Tsimpis for all our different discussions that led to the results of this article. I also would like to thank Daniele Faenzi for his help and interest in some of the questions of this work, Udhav Fowdar for our discussions about $G$-structures and Simon Salamon for his interest in this work.  I am partially supported by the grant BRIDGES ANR-FAPESP ANR-21-CE40- 0017.
\section{A twist reduction}
\subsection{An abstract version of the reduction process of Larfors,Tsimpis and Lüst}
Let $(X,h_X,J_X,\omega_X,\Omega_X)$ be a Kähler manifold of complex dimension $n+s$ with an $SU(n+s)$-structure given by $(\Omega_X,\omega_X)$. In particular, $\Omega_X$ is a nowhere-vanishing complex form of type $(n,0)$ and $\omega_X$ is a Kähler form of type $(1,1)$ and satisfy the equations (\ref{SU(n)eq1}) and (\ref{SU(n)eq2}). \\ \indent Consider a free, Hamiltonian and holomorphic action of a compact torus $T^s$ on $X$ with moment map $\mu : X\to\tor^*$ where $\tor :=Lie(T^s)$. For $a\in\tor$, we let $V_a$ to be the vector field induced on $X$ by the action: \[V_a(p):=\frac{d}{dt}\bigg|_{t=0}\exp(ta)\cdot p.\]
The Hamiltonian condition reads as $\mu$ being invariant by the action and \[\forall a\in\tor, \iota_{V_a}\omega_X=-d\langle\mu,a\rangle.\]
\indent Let $c$ be a regular value of $\mu$ and $M=\mu^{-1}(c)$ the corresponding level set. We denote by $Y=M/\!\!/T^s$ the quotient manifold which is a manifold of real dimension $2n$. For any differential form $\Phi_X$ on $X$, we write $\Phi_M=i^*\Phi_X$ for the induced form on $M$ where $i : M\hookrightarrow X$ is the inclusion map. A form $\Phi_X$ descends to the quotient $Y$ if there exists a form $\Phi_Y$ on $Y$ such that $\Phi_M=\pi^*\Phi_Y$ with $\pi : M\to Y$ the projection map. Recall that $\Phi_X$ descends to the quotient $Y$ iff $\Phi_M$ is basic on $M$ for the $T^s$-action , i.e \[\forall a\in\tor,\iota_{V_a}\Phi_M=\lied_{V_a}\Phi_M=0,\] where $V_a$ is a vector field on $M$ when restricted to it. Fixing a basis $e_1,\cdots,e_s$ of $\tor$ and setting $V_k=V_{e_k}$ the corresponding vector fields, the property of being basic then needs just to be checked on the $V_k$.\\
\indent The strategy in \cite{Larfors_2010}, which we generalize here, consists in manipulating algebraically the forms $\Omega_X,\omega_X$ to get basic forms on $M$ which satisfy the same kind of equations as (\ref{SU(n)eq1}) and (\ref{SU(n)eq2}). We note:
\begin{lemma}
    The Kähler form $\omega_X$ descends to the quotient $Y$.
\end{lemma}
\begin{proof}
    This is an application of the symplectic reduction theorem.
\end{proof}
Note that $i^*\Omega_X=0$ for degree reasons and we shall first reduce its degree. The idea is to work in the complexified tangent space of $X$, $TX^\C$, and to take the $(1,0)$-parts of the vector fields $J_XV_k$, namely we set
\[\xi_k:=-\frac{1}{2}(J_XV_k+iV_k),\] for all $k=1,...,s$ and we define $\Omega_M:=(\iota_{\xi_1}\cdots\iota_{\xi_s}\Omega_X)_M$. 
\begin{prop} $\Omega_M$ and $\omega_M$ satisfy the following equations:
    \begin{align}
        \Omega_M\wedge\omega_M&=0, \label{eqintermediaire1} \\
        \Omega_M\wedge\overline{\Omega_M}&=c_n\|V_1\|^2\cdots\|V_s\|^2\omega_M^n.\label{eqintermediaire2}
    \end{align}
\end{prop}
\begin{proof}
    Using $\Omega_X\wedge\omega_X=0$ and contracting successively by $\xi_1,\cdots\xi_s$, we obtain equation (\ref{eqintermediaire1}). For equation  (\ref{eqintermediaire2}), start from $\Omega_X\wedge\overline{\Omega_X}=c_{n+s}\omega_X^{n+s}$ and contract with $\xi_1,\overline{\xi_1},\cdots,\xi_s,\overline{\xi_s}$. Noticing that $\iota_{\overline{\xi_s}}\Omega_X=0$ and $\omega_X(\xi_k,\overline{\xi_k})=\frac{i}{2}\|V_k\|^2$, we obtain the second equation.
\end{proof}
 \indent $\Omega_M$ is horizontal by construction but it is not basic since it has charge $q$. To make it basic, we perform a twist on $\Omega_M$, following the idea of \cite{Larfors_2013,Larfors_2010}. We assume the existence of a twist form, say $\alpha_X$, see Definition \ref{twistformdef}, of charge $\frac{q}{2}$.\\
 \indent We recall some basic facts about contracting with $1$-forms. The Hermitian metric $h_X$ on $X$ gives rise to an  anti-linear isomorphism $h_X : T^{1,0}X\to (T^{1,0}X)^*$ given by \[\forall\xi\in T^{1,0}X, h_X(\cdot,\overline{\xi})\in(T^{1,0}X)^*.\]\indent For a $(1,0)$-form $\Phi$, we denote by $\Phi^\sharp$ the $(0,1)$ vector field dual to $\Phi$ and by conjugacy for $(0,1)$-forms. Intuitively, the contraction by a $1$-form $\Phi$ corresponds to contracting by $\Phi^\sharp$ and this operation is denoted by $\Phi\cdot$.
 \begin{prop}
 Let \[\hat\Omega:=\overline{\alpha_M}\cdot\Omega_M,\,\hat\omega:=\omega_M-\frac{i}{2\|\alpha_M\|^2}\alpha_M\wedge\overline{\alpha_M},\] where $\|\alpha_M\|^2:=\overline{\alpha_M}\cdot\alpha_M$ is the norm of $\alpha_M$.
 $\hat\Omega$ and $\hat\omega$ satisfy the following equations:
     \begin{align}
             \hat\Omega\wedge\hat\omega&=0, \label{eqintermediate3}\\
         \hat\Omega\wedge\overline{\hat\Omega}&=c_{n-1}\|\alpha_M\|^2\|V_1\|^2\cdots\|V_s\|^2\hat\omega^{n-1}.\label{eqintermediate4} 
     \end{align}
 \end{prop}
 \begin{proof}
     We first note that $\overline{\alpha_X}\cdot\omega_X=\frac{i}{2}\overline{\alpha_X}$ (resp. $\alpha_X\cdot\omega_X=-\frac{i}{2 }\alpha_X)$. Using equations (\ref{eqintermediaire1}), (\ref{eqintermediaire2}), the horizontality condition of $\alpha_X$ and $\Omega_X\wedge\alpha_X=0$, we obtain equations (\ref{eqintermediate3}) and (\ref{eqintermediate4}) by expanding the expressions. 
 \end{proof}
 \indent Note that $\alpha_M\wedge\overline{\alpha_M}$ is basic and $\|\alpha_M\|^2$ is $T^s$-invariant as the action is Hamiltonian and holomorphic. As a result, $\hat\omega$ is basic on $M$. For $\hat\Omega$, we first remark that \[\forall k, \lied_{V_k}\overline{\alpha_M}^\sharp=-\frac{q(V_k)}{2}\overline{\alpha_M}^\sharp.\] This comes from the fact that $\overline{\alpha_M}$ is of charge $-\frac{q}{2}$ and the $V_k$ are Killing vector fields. Using properties of the Lie derivative, we get that $\hat\Omega$ is of charge $\frac{q}{2}$. \\ 
 \indent With all these ingredients, we are able to give a proof of Theorem \ref{twistsimpletheo}: 
 \begin{proof}
      Note that the $\Omega,\omega$ of the statement of the theorem are \[\Omega=\overline{\alpha_M}\wedge\hat\Omega,\omega=\hat\omega-\frac{i}{2\|\alpha_M\|^2}\alpha_M\wedge\overline{\alpha_M}.\] By our previous computations, it is straightforward to see that $\Omega,\omega$ are both basic on $M$. The equation (\ref{eq1theo2}) is also immediate by using (\ref{eqintermediate3}) and the fact that $\overline{\alpha_M}$ is a $1$-form. We detail the computations of (\ref{eq1theo}):
      \begin{align*}
          \Omega\wedge\overline{\Omega}&=\overline{\alpha_M}\wedge\hat\Omega\wedge\alpha_M\wedge\overline{\hat\Omega}\\
          &=(-1)^n\alpha_M\wedge\overline{\alpha_M}\wedge\hat\Omega\wedge\overline{\hat\Omega}\\
          &=(-1)^nc_{n-1}\|\alpha_M\|^2\|V_1\|^2\cdots\|V_s\|^2\alpha_M\wedge\overline{\alpha_M}\wedge\hat\omega^{n-1}
      \end{align*}
      By equations (\ref{eqintermediate3}) and (\ref{eqintermediate4}), we get $\hat\omega^n=0$ 
      and thus $\alpha_M\wedge\overline{\alpha_M}\wedge\hat\omega^{n-1}=\frac{2i\|\alpha_M\|^2}{n}\omega^n$. Plugging this in the last line of the computation of $\Omega\wedge\overline{\Omega}$ and using the expression for $c_n$, we get equation (\ref{eq1theo}).
 \end{proof}
  Corollary \ref{cor1} is immediate since $\|\alpha_M\|^2,\|V_k\|^2$ are $T^s$-invariant functions. Observe that even if $\alpha$ has zeros, $\omega$ remains bounded. In that case, we will get an $SU(n)$-structure on the manifold $Y$ with singularities.
\subsection{A multiple twist theorem}
In this section, we prove Theorem \ref{multipletwisttheo} which is a repetitive application of the previous process. To make it work, we need to define what orthogonal forms are:
\begin{Def}\label{orthogonaldef}
    Two $1$-forms $\alpha,\beta$ on $X$ are orthogonal if \[\overline{\alpha}\cdot\beta=\overline{\beta}\cdot\alpha=0.\]
\end{Def}
\indent We now give a proof of Theorem \ref{multipletwisttheo} when $l=2$, for general $l$, it is just an induction of this special case.
\begin{proof}
    Let $\alpha_X,\beta_X$ be two orthogonal twist forms on $X$ of respective charges $k_\alpha,k_\beta$ such that $k_\alpha+k_\beta=\frac{q}{2}$. Set \[\hat\Omega':=\overline{\alpha_M}\cdot\Omega_M,\,\hat\omega':=\omega_M-\frac{i}{2\|\alpha_M\|^2}\alpha_M\wedge\overline{\alpha_M} ,\]\[\hat\Omega:=\overline{\beta_M}\cdot\hat\Omega',\,\omega':=\hat\omega'-\frac{i}{2\|\beta_M\|^2}\beta_M\wedge\overline{\beta_M}.\]
    All the computations done in Section $2.1$ are valid for $(\hat\Omega',\hat\omega')$. We first prove similar equations for $(\hat\Omega,\hat\omega)$ to equations (\ref{eqintermediate3}) and (\ref{eqintermediate4}). Contracting equation (\ref{eqintermediate3}) by $\overline{\beta_M}$ gives :
    \begin{equation}\label{eq1proof}
        \hat\Omega\wedge\hat\omega'+(-1)^{n-1}\hat\Omega'\wedge\overline{\beta_M}\cdot\hat\omega'=0.
    \end{equation}
    By the orthogonality assumption, we have \[\overline{\beta_M}\cdot\hat\omega'=\overline{\beta_M}\cdot\omega_M=\frac{i}{2}\overline{\beta_M}.\]
    Plugging this into equation (\ref{eq1proof}), we get 
    \begin{equation}\label{eq2proof}
        \hat\Omega\wedge\bigg(\hat\omega+\frac{i}{2\|\beta_M\|^2}\beta_M\wedge\overline{\beta_M}\bigg)+\frac{i}{2}(-1)^{n-1}\hat\Omega'\wedge\overline{\beta_M}=0.
    \end{equation}
    Note that $\hat\Omega\wedge\beta_M=(-1)^n\|\beta_M\|^2\hat\Omega'$ and with equation (\ref{eq2proof}) it follows that: \begin{equation}\label{eq3proof}
        \hat\Omega\wedge\hat\omega=0.
    \end{equation}
    Using equation (\ref{eqintermediate4}) and contracting it by $\beta_M$ and $\overline{\beta_M}$ eventually gives:
    \begin{equation}\label{eq4proof}
        \hat\Omega\wedge\overline{\hat\Omega}=c_{n-2}\|\alpha_M\|^2\|\beta_M\|^2\|V_1\|^2\cdots\|V_s\|^2\hat\omega^{n-2}.
    \end{equation}
    A simple computation also shows that $\hat\Omega$ is charge definite of charge $\frac{q}{2}$.\\ \indent Let \[\Omega=\overline{\beta_M}\wedge\overline{\alpha_M}\wedge\hat\Omega\quad\text{and}\quad\omega=\hat\omega-\frac{i}{2\|\alpha_M\|^2}\alpha_M\wedge\overline{\alpha_M}-\frac{i}{2\|\beta_M\|^2}\beta_M\wedge\overline{\beta_M}\] be the forms of the statement. They are basic on $M$ since $k_\alpha+k_\beta=\frac{q}{2}$, equation (\ref{eq1theo2}) is a direct consequence of equation (\ref{eq3proof}). Equation (\ref{eq2theo2}) comes from equation (\ref{eq4proof}) and from the following remark:
    \[\omega^n=-\frac{n(n-1)}{\|\alpha_M\|^2\|\beta_M\|^2}\alpha_M\wedge\overline{\alpha_M}\wedge\beta_M\wedge\overline{\beta_M}\wedge\hat\omega^{n-2}.\]
\end{proof}
The orthogonality condition seems quite restrictive. Note that we could replace this condition by the freeness of the family at every point. Starting from such a family, we can produce using a Gram-Schmidt process an orthogonal family, see Section $3.2$.
\section{Geometric applications}
\subsection{Twist forms on $(\C^N,(\C^*)_{diag})$}
    This section is motivated by the example of $\C\p^3$ in \cite{Larfors_2013,Larfors_2010} where the authors find an explicit twist form on $\C^4$ and also compute the torsion of the constructed $SU(3)$-structure. We review this example and also find a family of twist forms of charge $2$. Using this family, we provide a whole family of $SU(3)$-structures on $\C\p^3$. We also prove that on $\C^N\backslash\{0\}$, holomorphic twist forms of polynomial type of charge greater than $2$ do not exist. \\
    \indent Consider $\C^N$ with its canonical Kähler structure with Kähler form \[\omega_0:=\frac{i}{2}\sum_{j=1}^Ndz_j\wedge d\overline{z_j}.\]\indent We also let $\Omega_0$ be the standard holomorphic volume form on $\C^N$ given by \[\Omega_0:=dz_1\wedge\cdots\wedge dz_N.\] \indent $(\Omega_0,\omega_0)$ defines the standard $SU(N)$-structure on $\C^N$. Theorem \ref{twistsimpletheo} relies on the existence of twist forms. It is natural to wonder whether they exist and we first study this question on $\C^N$ with the standard action of $T^1=S^1$.\newpage \indent Let $\xi:=\sum_{j=1}^Nz_j\partial_j$ be the holomorphic Euler vector field introduced in Section $2.1$ for this action. We restrict our attention to holomorphic twist forms, so that only $\xi$ matters in the computations. Let \[\alpha:=\sum_{j=1}^N\alpha_jdz_j\] be a holomorphic twist form of integer charge $k=k_\alpha(\xi)\geq 1$. Since the functions $z_j$ are of charge $1$, we get that the holomorphic functions $\alpha_j$ must be of charge $k-1$. But, this reads as \[\sum_{l=1}^Nz_l\frac{\partial\alpha_j}{\partial z_l}=(k-1)\alpha_j,\] and forces $\alpha_j$ to be homogeneous polynomials of degree $k-1$. This shows that any holomorphic twist form on $\C^N$ with its diagonal circle action must be of polynomial type. \\ \indent Rewriting $\alpha=\sum_{j=1}^NP_jdz_j$ with $P_j$ homogeneous polynomials of degree $k-1$, the horizontality condition is \[\forall z=(z_1,\cdots,z_N)\in\C^N, \sum_{j=1}^NP_j(z)z_j=0.\]
    \indent The vanishing points of $\alpha$ are then the common zeros of the $P_j$ and as they are homogeneous, at least $0$ is a vanishing point of $\alpha$. Thus, we will require that $\alpha$ is nowhere-vanishing on $\C^N\backslash\{0\}$. \\
    \indent We first focus on the case $k=2$ and we show the classification result given by the Proposition \ref{su(3)cp3}. We introduce the fundamental forms \[\alpha_{ij}:=z_idz_j-z_jdz_i,\] for $1\leq i,j\leq N$. These fundamental forms are horizontal and charge definite of charge $2$. Using these fundamental forms, we can prove Proposition \ref{su(3)cp3}:
    \begin{proof}
        If $M=(m_{ij})_{1\leq i,j\leq N}$ is a skew-symmetric invertible $N\times N$ complex matrix and $\alpha_M:=\sum_{i<j}m_{ij}\alpha_{ij}$ is horizontal and charge definite of charge $2$ since the $\alpha_{ij}$ satisfy these properties. By construction, $\alpha_M$ vanishes on the kernel of $M$ which is zero by assumption. Conversely, if $\alpha$ is a holomorphic twist form of charge $2$ on $\C^N\backslash\{0\}$, then we already know that $\alpha$ must be of the form \[\alpha={}^tZMdZ,\] where $M$ is some $N\times N$ complex matrix. Note that $M$ is invertible since $\alpha$ is nowhere-vanishing on $\C^N\backslash\{0\}$. Eventually, the horizontality condition forces $M$ to be skew-symmetric. 
    \end{proof}
\noindent\textbf{Remark 3.}
        Note that this also proves that on even dimensional $\C^{2n+1}\backslash\{0\}$, there are no holomorphic twist forms of charge $2$. \\ 
    \indent Theorem \ref{cp3} is then just an application of both Proposition \ref{su(3)cp3} and Theorem \ref{twistsimpletheo}. Note that for \[M=\begin{pmatrix}
        J&0\\0&J
    \end{pmatrix},\] where $J=\begin{pmatrix}
        0&1\\-1&0
    \end{pmatrix}$, the corresponding $\alpha_M$ is the twist form originally used in \cite{Larfors_2013,Larfors_2010}. In Section $4$, we will study more specifically the geometry of the $SU(3)$-structures resulting from Theorem \ref{cp3}.\\ \\
    \noindent  \textbf{Remark 4.} We have chosen here to work with the diagonal action of $S^1$, but we could have worked with other actions of $S^1$ on $\C^4$ and have $SU(3)$-structures on weighted complex projective spaces. More precisely, if $k=(k_1,k_2,k_3,k_4)\in\Z^4$ denotes the charge of the weighted $S^1$ action, we can slightly modify the forms given by Proposition \ref{su(3)cp3} to make them twist forms for $\C\p^3_k$.
   \\\\ \indent We now turn our attention to twist forms of charge $k>2$ to see if we can apply Theorem \ref{twistsimpletheo} such as we did for $\C\p^3$. However, it turns out that no holomorphic twist forms on $\C^N\backslash\{0\}$ of charge $k>2$ exist. More precisely, we show that: 
   \begin{prop}\label{nonexistence}
       Let $(P_j)_{1\leq j\leq N}$ be $N$ homogeneous complex polynomials of degree $k\geq 2$ in $N$ variables. If \[\forall z=(z_1,\cdots,z_N)\in\C^N,\sum_{i=1}^Nz_iP_i(z)=0,\] then the $P_j$ have a non trivial common root.  
   \end{prop}
\begin{proof}
 Homogeneous polynomials of $N$ variables of degree $k$ are precisely holomorphic sections of the line bundle $\Or(k)$ over $\C\p^{N-1}$. By duality, we can see the $P_j$ as morphisms from $\Or(-k)$ to $\Or$ the trivial line bundle. Assume the $P_j$ have only zero as common root. Then, the map $P_1\oplus\cdots\oplus P_N : \Or(-k)^N\to\Or$ is surjective. Consider its kernel $K$. The global sections of $K$ consist of homogeneous polynomials, say $Q_1,\cdots,Q_N$, such that \[\sum_iQ_iP_i=0.\] If $d$ is the degree of the $Q_i$, the global sections are sections of $K(d)=K\otimes\Or(d)$. We have the following exact sequence \[0\longrightarrow K\longrightarrow\Or(-k)^{\oplus N}\longrightarrow \Or\longrightarrow 0\]  Tensoring with $\Or(d)$ gives \[0\longrightarrow K(d)\longrightarrow\Or(-k+d)^{\oplus N}\longrightarrow\Or(d)\longrightarrow 0\] In particular, it induces the following sequence \[0\longrightarrow H^0(K(d))\longrightarrow H^0(\Or(-k+d))^{\oplus N}\] It is well-known that if $-k+d<0$, then $H^0(\Or(-k+d))=0$ and thus $H^0(K(d))=0$ for $d<k$. In conclusion if $k\geq 2$, then the $Z_i$ cannot be in $K$.
\end{proof}
In summary, we have proved that for the diagonal $S^1$ action on $\C^{n+1}$, the only holomorphic twist forms are those of charge $2$ given by Proposition \ref{su(3)cp3}.
\subsection{Case of $\C\p^{4n-1}$}
Proposition \ref{nonexistence} shows that we cannot directly apply Theorem \ref{twistsimpletheo} with holomorphic twist forms for complex projective spaces of higher dimension than $3$. This leads us to use Theorem \ref{multipletwisttheo} instead. \\ 
\indent We go back to the $\C^N$ case of Section $3.1$ with $N=4n$. We let \[\mathcal{T}:=\{\alpha_M, M\in GL(N,\C),\, {}^tM=-M\}\] be the set of holomorphic twist forms we have found in the previous section. Note that two forms of $\mathcal{T}$ are not orthogonal in general. Our first goal is to produce an orthogonal family of twist forms out of $\mathcal{T}$.\\ \indent We outline the Gram-Schmidt process to obtain such a family. 
\begin{Def}
    Let $\alpha$ be a nowhere-vanishing $1$-form. We define the orthogonal projector of $\alpha$, $P_\alpha$ as \[\forall\Phi\in T^*\C^N, P_\alpha(\Phi)=\Phi-\frac{\overline{\alpha}\cdot\Phi}{\|\alpha\|^2}\alpha.\] 
\end{Def}
The following lemma describes how the orthogonal projectors behave with $\mathcal{T}$: 
\begin{lemma}\label{gramschmidt}
    Let $\alpha,\beta$ be two forms of $\mathcal{T}$. We have :
    \begin{enumerate}[i)]
        \item $P_\alpha(\beta)$ is a $(1,0)$-form orthogonal to $\alpha$.
        \item $P_\alpha(\beta)$ is horizontal, charge definite of charge $2$.
        \item The norm of $P_\alpha(\beta)$ is \[\|P_\alpha(\beta)\|^2=\|\beta\|^2-\frac{|\overline{\alpha}\cdot\beta|^2}{\|\alpha\|^2}.\]
    \end{enumerate}
    In particular, if $\alpha$ and $\beta$ are linearly independent, then $P_\alpha(\beta)$ is nowhere-vanishing. 
\end{lemma}
\begin{proof}
    Points $i)$ and $iii)$ are straightforward. $P_\alpha(\beta)$ is horizontal and is of charge $2$ as $\frac{\overline{\alpha}\cdot\beta}{\|\alpha\|^2}$ is invariant.
\end{proof}
This process extends to many forms of $\mathcal{T}$ by induction in the same way as for the usual Gram-Schmidt process. However, there is no subfamily of $n\geq 2$ linearly independent forms of $\mathcal{T}$. In fact, if we use any subfamily of $n$ forms of $\mathcal{T}$, the Gram-Schmidt process will produce zeros and so singularities of the $SU(4n-1)$-structure. For general $n$, it is hard to know where these singularities might appear, but for $n=2$, we have the following lemma:
\begin{lemma}
    Let $\alpha_1,\alpha_2$ be two generic twist forms of $\mathcal{T}$ on $\C^8\backslash\{0\}$. $\alpha_1,\alpha_2$ are linearly dependent at four different complex planes in $\C^8$.
\end{lemma}
\begin{proof}
    Denote $M_1,M_2$ the corresponding invertible and skew-symmetric matrices. $\alpha_1,\alpha_2$ are collinear at $z\in\C^8\backslash\{0\}$ iff $z$ is an eigenvector of the matrix $M_1^{-1}M_2$. If $M_1,M_2$ are chosen generically, the matrix $M_1^{-1}M_2$ has four distinct eigenvalues with algebraic multiplicity equal to $2$.
\end{proof}
This lemma together with Theorem \ref{multipletwisttheo} shows the first part of Theorem \ref{cp7}. For the singularities, take two generic matrices $M_1,M_2$ of the form $M_1=\diag(J,J,J,J)$ and $M_2=\diag(\lambda_1J,\lambda_2J,\lambda_3J,\lambda_4J)$ with complex $\lambda_i$ all different and non-zero. Denote $r(z)=|z|^2$ and $r_\lambda(z)=\lambda_1(|z_1|^2+|z_2|^2)+\lambda_2(|z_3|^2+|z_4|^2)+\lambda_3(|z_5|^2+|z_6|^2)+\lambda_4(|z_7|^2+|z_8|^2)$ so that the Gram-Schmidt process gives:\[\beta={}^tZ\bigg(M_2-\frac{r_\lambda(z)}{r(z)}M_1\bigg)dZ\] Introduce $\rho_i(z)=\lambda_i-\frac{r_\lambda(z)}{r(z)}$ to write concisely:\[\beta={}^tZ\diag(\rho_1(z)J,\rho_2(z)J,\rho_3(z)J,\rho_4(z)J)dZ\] In particular, we have:
\[\|\beta\|^2=|\rho_1|^2(|z_1|^2+|z_2|^2)+|\rho_2|^2(|z_3|^2+|z_4|^2)+|\rho_3|^2(|z_5|^2+|z_6|^2)+|\rho_4|^2(|z_7|^2+|z_8|^2) \] Then, we have:
\begin{align*}
    \beta\wedge\bar\beta&=\sum_{i,j}\rho_i\bar\rho_j(z_{2i-1}dz_{2i}-z_{2i}dz_{2i-1})\wedge(\bar z_{2j-1}d\bar z_{2j}-\bar z_{2j}d\bar z_{2j-1})\\&=\sum_{i,j}\rho_i\bar\rho_j(z_{2i-1}\bar z_{2j-1}dz_{2i}\wedge d\bar z_{2j}-z_{2i-1}\bar z_{2j} dz_{2i}\wedge d\bar z_{2j-1}-z_{2i}\bar z_{2j-1}dz_{2i-1}\wedge d\bar z_{2j}\\&+z_{2i}\bar z_{2j}dz_{2i-1}\wedge d\bar z_{2j-1})     
\end{align*} We see that the zeros of $\beta$ are where only one $\rho_i$ vanishes.
Take $Z_1\in\C^2$ of norm $1$ and consider the sequences $Z_{\varepsilon_1}=(Z_1,(\varepsilon,0),0,0)$ and $Z_{\varepsilon_2}=(Z_1,(0,\varepsilon),0,0)$. A long computation shows that:
\[\frac{\beta\wedge\bar\beta}{\|\beta\|^2}(Z_{\varepsilon_1})\underset{\varepsilon\to 0}{\to}dz_4\wedge d\bar z_4,\,\frac{\beta\wedge\bar\beta}{\|\beta\|^2}(Z_{\varepsilon_2})\underset{\varepsilon\to 0}{\to}dz_3\wedge d\bar z_3\] which proves that the limit does not exist. 
\subsection{Revisiting $SU(3)$-structures on toric $\C\p^1$-bundles}
    In \cite{Larfors_2010}, Larfors examines toric $\C\p^1$-bundles over the Hirzebruch surfaces $\F_n$. \\ \indent We start by recalling basic facts and definitions about these surfaces, following \cite{hirzebruchgauduchon}. There are several ways of defining Hirzebruch surfaces, here we focus on the toric description. 
    \\\indent Consider the space $Y=\C^2\backslash\{0\}\times\C^2\backslash\{0\}$ together with an action $\sim_n$ of the compact torus $T^2$, where $n\in\Z$, given by :
    \[\forall (\lambda,\mu)\in T^2,\forall z=(z_0,z_1)\times(z'_0,z'_1)\in Y, (\lambda,\mu)\cdot z=(\lambda z_0, \lambda z_1)\times(\lambda^n\mu z'_0,\mu z'_1).\]
 Then, the Hirzebruch surface $\F_n$ is $Y/\sim_n$. $\F_n$ is then by construction  a toric surface. \\ \indent We now consider $X$ a toric $\C\p^1$-bundle over $\F_n$. From Delzant theory, see e.g \cite{calderbank2003guilleminformulakaehlermetrics}, we can describe $X$ as a toric quotient of $\C^2\backslash\{0\}\times\C^2\backslash\{0\}\times \C^2\backslash\{0\}$ by $T^3$ given by the charges :
 \[q^1=(1,1,n,0,0,a), \, q^2=(0,0,1,1,0,b),\, q^3=(0,0,0,0,1,1),\] 
where $a,b\in\Z$. \\ \indent Note that the first four components of the charges correspond to the charges of $\F_n$. The parameters $a,b$ describe how twisted $X$ is in $\C\p^1$ and $q^3$ corresponds to the $\C\p^1$ fiber of $X$. \\ \indent We thus work in $\C^6$ with its usual $SU(6)$-structure. The $T^3$ action on $\C^6$ generates the holomorphic vector fields:
\[\xi_1:=z_1\partial_1+z_2\partial_2+nz_3\partial_3+az_6\partial_6, \xi_2:=z_3\partial_3+z_4\partial_4+bz_6\partial_6, \xi_3:=z_5\partial_5+z_6\partial_6,\] which gives the charge of $\Omega_0=dz_1\wedge\cdots\wedge dz_6$: \[q(\xi_1)=2+n+a,q(\xi_2)=2+b, q(\xi_3)=2.\] The moment map is \[\mu(z):=\bigg(|z_1|^2+|z_2|^2+n|z_3|^2+a|z_6|^2,|z_3|^2+|z_4|^2+b|z_6|^2,|z_5|^2+|z_6|^2\bigg).\]
\indent In what follows, we assume that $a=2-n,b=-2$. On a positive regular value, set, as in \cite{Larfors_2010}, \[\alpha_1:=-z_2dz_1+z_1dz_2, \alpha_2:=-z_4dz_3+z_3dz_4+\frac{nz_3z_4}{|z_1|^2+|z_2|^2}(\overline{z_1}dz_1+\overline{z_2}dz_2)\]
\indent A simple computation shows that \[k_{\alpha_1}(V_1)=2i, k_{\alpha_1}(V_2)=0, k_{\alpha_1}(V_3)=0\] \[k_{\alpha_2}(V_1)=in, k_{\alpha_2}(V_2)=2i, k_{\alpha_2}(V_3)=0\]\indent Letting \[\alpha:=z_5\alpha_1+z_6\alpha_2\] we get a non-holomorphic twist form and therefore Theorem \ref{twistsimpletheo} can be applied to get an $SU(3)$-structure on $X$.
\subsection{Calabi-Yau cones over Fano Kähler-Einstein manifolds} In this paragraph, we use the theory of Fano Kähler-Einstein manifolds to provide a setting where Theorem \ref{twistsimpletheo} can be applied. We first recall the main facts about Fano Kähler-Einstein manifolds:
\begin{Def}
    An almost complex manifold $(Y,J)$ is called Fano if: \[c_1(Y,J)=c_1(-K_Y)>0,\] where $K_Y$ is the canonical line bundle and $-K_Y$ the anticanonical line bundle. Alternatively, $-K_Y$ is an ample line bundle.
\end{Def}
The most important result is the following, see e.g \cite{sparks2010sasakieinsteinmanifolds}:
\begin{prop}
   Let $(Y,J,\omega)$ be a Fano Kähler-Einstein manifold. Then, $K_Y\backslash Y$ is a Calabi-Yau cone. 
\end{prop}
More explicitly, the holomorphic volume form is given locally by: \[\Omega_{K_Y}=d(ydz_1\wedge\cdots\wedge dz_n),\] where the $(z_k)$ are holomorphic coordinates on $Y$ and $y$ the fiber coordinate. We can now give a proof of Theorem \ref{cones}:
\begin{proof}
    Recall that the Fano index of $Y$ is defined as the maximal integer $r$ such that there exists a holomorphic line bundle $L$ satisfying $-K_Y=L^{\otimes r}$. In what follows, we assume that $r$ is a multiple of $4$ so that we can write $-K_Y=L^{\otimes 4}$ with $L$ a very ample line bundle. Denote $\pi : z\in L^{-1}\backslash Y\mapsto z^{\otimes 4}\in K_Y\backslash Y$ and consider the fiberwise $S^1$ action on $L^{-1}\backslash Y$. This action extends to $K_Y\backslash Y$ to make $\pi$ an $S^1$-equivariant map. Using the local formula of $\Omega_{K_Y}$, we see that $\Omega_{K_Y}$ is of charge $4$ for this action and thus $\Omega=\pi^*\Omega_{K_Y}$ is a holomorphic volume form on $L^{-1}\backslash Y$ of charge $4$ with respect to the fiberwise action of $S^1$ and pulling back the Hermitian structure of $K_Y\backslash Y$, we have a Calabi-Yau structure on $L^{-1}\backslash Y$ for which the fiberwise $S^1$ action is Hamiltonian. Kodaira's embedding theorem gives a holomorphic $S^1$-equivariant map $\Phi : L^{-1}\backslash Y\hookrightarrow\C^N\backslash\{0\}$ where $N=\dim H^0(Y,\Or(L))$. Without loss of generality, we can assume $N$ to be even and if $\beta$ is a holomorphic twist form of charge $2$ on $\C^N\backslash\{0\}$, its pullback $\alpha$ defines a $(1,0)$-form which is horizontal and of charge $2$ for the fiberwise action of $S^1$ on $L^{-1}\backslash Y$. Theorem \ref{twistsimpletheo} applies and we obtain an $SU(n)$-structure on $Y$ with possible singularities. Eventually, this also works by taking any smooth section of the vector bundle $(T^{1,0}Y)^*\otimes L^{\otimes 2} $.
\end{proof}
\textit{Application :} This gives a new perspective on how to get an $SU(4n-1)$-structure on $\C\p^{4n-1}$. Indeed, if $X=\C\p^{4n-1}$, it is well-known that $-K_X=\Or(4n)=\Or(n)^{\otimes 4}$. The Kodaira embedding theorem then gives a map $\varphi : \Or(-n)\backslash\C\p^{4n-1}\to\C^{2m}\backslash\{0\}$ where $2m=\dim H^0(X,\Or(n))=\binom{5n-1}{n}$. $\varphi$ is explicitly given by: \[\varphi([z_0,\cdots,z_{4n-1}])=[z_0^n,z_0^{n-1}z_1,\cdots,z_{4n-1}^n],\]  the components of the image of $\varphi$ are precisely the homogeneous monomials of degree $n$. Theorem \ref{cones} applies and provides to any $\C\p^{4n-1}$ an $SU(4n-1)$-structure. In particular for $n=1$, we exactly get back the $SU(3)$ structures of Theorem \ref{cp3}. 
\section{Torsion of $SU(3)$-structures on $\C\p^3$} 
In this section, we study the torsion of the $SU(3)$-structures on $\C\p^3$ given by Theorem \ref{cp3}. We first recall that if $(X,\omega,\Omega)$ is a manifold with an $SU(3)$-structure, its torsion is characterized by: \[d\omega=\frac{3}{2}\Ima(\overline{W_1}\Omega)+W_4\wedge\omega+W_3,\] \[d\Omega=W_1\omega^2+W_2\wedge\omega+\overline{W_5}\wedge\Omega,\] where $W_1$ is a smooth complex function, $W_2$ is a primitive $(1,1)$ form, $W_3$ is a real primitive $(1,2)+(2,1)$ form, $W_4$ is a real $(1,0)+(0,1)$ form and $W_5$ is a complex $(1,0)$ form. The $W_i$ are known to be the torsion classes and were first described by Gray and Hervella in \cite{Grayhervella}. Each $W_i$ contributes to a special feature of the $SU(3)$-structure, e.g, the integrability of the complex structure is entirely described by $W_1$ and $W_2$. \\
\indent Let $(\C^4,\omega_0,\Omega_0)$ be the canonical $SU(4)$-structure on $\C^4$. Let $M$ be a complex skew-symmetric and invertible $4\times 4$ complex matrix and $\alpha_M$ the corresponding holomorphic twist form introduced in the Propositon \ref{su(3)cp3}. The $SU(3)$-structure on $\C\p^3$ is given by \[\Omega_M:=\frac{1}{\|\alpha_M\|^2}\overline{\alpha_M}\wedge\hat\Omega_M, \omega_M:=\omega_0-\frac{i}{\|\alpha_M\|^2}\alpha_M\wedge\overline{\alpha_M}.\] The key observation to computing the torsion classes is: \begin{lemma} Let $M$ be a skew-symmetric and invertible $4\times 4$ complex matrix and $\alpha_M$ the corresponding holomorphic twist form. 
    Then, there exists $\lambda_M$ a smooth complex function and $\eta_M$ a $(1,0)$-form such that \[d\alpha_M=\lambda_M\hat\Omega_M+\alpha_M\wedge\eta_M.\]
\end{lemma}
\begin{proof}
    $\alpha_M$ being holomorphic, $d\alpha_M$ is of type $(2,0)$. A simple computation shows that $d\alpha_M\wedge d\alpha_M=8Pf(M)\Omega_0$ where $Pf(M)$ is the Pfafian of $M$. Contracting by $\xi$ and $\overline{\alpha_M}$ gives $\lambda_M=\frac{2Pf(M)}{\|\alpha_M\|^2}$ and $\eta_M=\frac{\overline{\alpha_M}\cdot d\alpha_M}{\|\alpha_M\|^2}$.
\end{proof}
In particular, we note that $\lambda_M$ is nowhere zero. With this lemma, we can give a proof of Theorem \ref{torsion}:
\begin{proof}
    Observe that \[d\Omega_M\wedge\omega_M=\Omega_M\wedge d\omega_M=W_1\omega_M^3.\] Using the previous lemma and the closeness of $\omega_0$, we get that \[\Omega_M\wedge d\omega_M=-i\overline{\lambda_M}\Omega_M\wedge\overline{\Omega_M}=\frac{4}{3}\overline{\lambda_M}\omega_M^3.\] Then, we compute $d\omega_M$ to get $W_4$ and $W_3$: \begin{align*}
        d\omega_M&=d(\omega_0-\frac{i}{\|\alpha_M\|^2}\alpha_M\wedge\overline{\alpha_M})\\
        &=-\frac{d\|\alpha_M\|^2}{\|\alpha_M\|^2}\wedge(\omega_M-\omega_0)+\frac{2}{\|\alpha_M\|^2}\Ima(d\alpha_M\wedge\overline{\alpha_M})\\
        &=-\frac{d\|\alpha_M\|^2}{\|\alpha_M\|^2}\wedge(\omega_M-\omega_0)+\frac{2}{\|\alpha_M\|^2}\Ima(\|\alpha_M\|^2\lambda_M\Omega_M-\alpha_M\wedge\overline{\alpha_M}\wedge\eta_M)\\
        &=-\frac{d\|\alpha_M\|^2}{\|\alpha_M\|^2}\wedge(\omega_M-\omega_0)+\frac{3}{2}\Ima(\overline{W_1}\Omega_M)-2\Ree(\eta_M)\wedge(\omega_M-\omega_0).
    \end{align*}
    A simple computation shows that \[-\frac{d\|\alpha_M\|^2}{\|\alpha_M\|^2}=\Ree(\eta_M).\] which simplifies $d\omega_M$ by \[d\omega_M=\frac{3}{2}\Ima(\overline{W_1}\Omega_M)-\Ree(\eta_M)\wedge\omega_M+\Ree(\eta_M)\wedge\omega_0,\] giving $W_4$ and $W_3$ as expected. Eventually, we compute $W_2$ and $W_5$: \begin{align*}
        d\Omega_M&=d(\frac{1}{\|\alpha_M\|^2}\overline{\alpha_M}\wedge\hat\Omega_M)\\
        &=-\frac{d\|\alpha_M\|^2}{\|\alpha_M\|^2}\wedge\Omega_M+\frac{1}{\|\alpha_M\|^2}(d\overline{\alpha_M}\wedge\hat\Omega_M-\overline{\alpha_M}\wedge d\hat\Omega_M)\\
        &=\Ree(\eta_M)\wedge\Omega_M+\frac{1}{\|\alpha_M\|^2}(2\overline{\lambda_M}\|\alpha_M\|^2\hat\omega_M^2-\|\alpha_M\|^2\overline{\eta_M}\wedge\Omega_M-\overline{\alpha_M}\wedge d\hat\Omega_M).\\
    \end{align*}
    By taking the differential of the equation of the lemma, we obtain that \[d\hat\Omega_M=-\Ree(\eta_M)\wedge\hat\Omega_M-\eta_M\wedge\hat\Omega_M+\frac{1}{\lambda_M}\alpha_M\wedge d\eta_M,\] and plugging this into $d\Omega_M$ allows us to write :
    \begin{align*}
        d\Omega_M&=\Ree(\eta_M)\wedge\Omega_M+2\overline{\lambda_M}\hat\omega_M^2-\overline{\eta_M}\wedge\Omega_M-\Ree(\eta_M)\wedge\Omega_M-\eta_M\wedge\Omega_M\\&+\frac{1}{\|\alpha_M\|^2\lambda_M}\alpha_M\wedge\overline{\alpha_M}\wedge d\eta_M\\
        &=-2\Ree(\eta_M)\wedge\Omega_M+2\overline{\lambda_M}(\omega_M^2+\frac{i}{\|\alpha_M\|^2}\alpha_M\wedge\overline{\alpha_M}\wedge\omega_M)\\&+\frac{1}{\lambda_M\|\alpha_M\|^2}\alpha_M\wedge\overline{\alpha_M}\wedge d\eta_M\\
      &=W_1\omega_M^2+\frac{2\overline{\lambda_M}}{3}\omega_M^2+\frac{2i\overline{\lambda_M}}{\|\alpha_M\|^2}\alpha_M\wedge\overline{\alpha_M}\wedge\omega_M\\&+\frac{1}{\lambda_M\|\alpha_M\|^2}\alpha_M\wedge\overline{\alpha_M}\wedge d\eta_M-2\Ree(\eta_M).
    \end{align*}
    We observe that $\alpha_M\wedge\overline{\alpha_M}$ and $d\eta_M$ are both of type $(1,1)$ which gives $W_2$ and consequently $W_5$. The structure is LT iff $\Ree(\eta_M)=0$ iff $\alpha_M$ has constant norm on the sphere $S^7$. The norm of $\alpha_M$ is precisely \[\|\alpha_M\|^2=z^*M^*Mz,\] and usual arguments in linear algebra shows that this is constant iff \[M^*M=\mu I_4,\] for some $\mu$ a real number. If $M$ is of this form, then \[d\eta_M=-i|\lambda_M|^2\omega_0,\] which proves the formula for $W_2$.
\end{proof}
\noindent\textbf{Remark 5.} It is well-known that $\C\p^3$ admits a nearly Kähler structure \cite{butruille2006homogeneousnearlykahlermanifolds}. Our computation shows that $W_2$ is never zero, an interesting question then would be to know if we could obtain the nearly Kähler structure of $\C\p^3$ by a similar approach than ours.
\printbibliography
\end{document}